\documentclass[article]{amsart}

 \newtheorem{thm}{Theorem}[section]
 \newtheorem{cor}[thm]{Corollary}
 \newtheorem{lem}[thm]{Lemma}
 
 \theoremstyle{definition}
 
 \theoremstyle{remark}
 \newtheorem{rem}[thm]{Remark}
 \theoremstyle{definition}

 \newcommand{\PP}{\mathbb{P}}

\def\move-in{\parshape=1.75true in 5true in}

\usepackage{xypic}

\begin{document}

\title{The Solution to Waring's Problem for Monomials}

\author[E. Carlini]{Enrico Carlini}
\address[E. Carlini]{Dipartimento di Matematica, Politecnico di Torino, Turin, Italy}
\email{enrico.carlini@polito.it}

\author[M.V.Catalisano]{Maria Virginia Catalisano}
\address[M.V.Catalisano]{Dipartimento di Ingegneria della Produzione, Termoenergetica e Modelli
Matematici, Universit\`{a} di Genova, Genoa, Italy.}
\email{catalisano@diptem.unige.it}

\author[A.V. Geramita]{Anthony V. Geramita}
\address[A.V. Geramita]{Department of Mathematics and Statistics, Queen's University, King\-ston, Ontario, Canada and Dipartimento di Matematica, Universit\`{a} di Genova, Genoa, Italy}
\email{Anthony.Geramita@gmail.com \\ geramita@dima.unige.it  }

\maketitle


\begin{abstract}
In the polynomial ring $T=k[y_1,\ldots,y_n]$, with $n>1$, we bound
the multiplicity of homogeneous radical ideals $I\subset
(y_1^{a_1},\ldots,y_n^{a_n})$ such that $T/I$ is a graded
$k$-algebra with Krull dimension one.  As a consequence we solve
the Waring Problem for all monomials, i.e. we compute the minimal
number of linear forms needed to write a monomial as a sum of
powers of these linear forms. Moreover, we give an explicit description of a sum of
powers decomposition for monomials. We also produce new bounds for the Waring
rank of polynomials which are a sum of pairwise coprime monomials.
 \end{abstract}

\section{Introduction}

Let $k$ be an algebraically closed field of characteristic zero and
 $k[x_1,\ldots,x_n]$ the standard graded polynomial ring in $n$ variables. Given a degree $d$ form $F$ the {\it Waring Problem for Polynomials} asks for the least value of $s$ for which there exist linear forms $L_1,\ldots,L_s$ such that

\[F=\sum_1^s L_i^d \ . \]
This value of $s$ is called the {\it Waring rank} of $F$ (or simply the {\it rank} of $F$) and will be denoted by $\mathrm{rk}(F)$.

There was a long-standing conjecture describing the rank of a generic form $F$ of degree $d$, but the verification of that conjecture was only found relatively recently in the famous work of J. Alexander and A. Hirschowitz \cite{AH95}.  However, for a given {\it specific} form $F$ of degree $d$ the value of $\mathrm{rk}(F)$ is not known in general. Moreover, there is no direct algorithmic way to compute the rank of a given form.
Given this state of affairs, several attempts have been made to compute the rank of specific forms.  One particular family of examples that has attracted attention is the collection of monomials.

A few cases where the ranks of specific monomials are computed can be found in \cite{LM} and in \cite{LandsbergTeitler2010}. The most complete result in this direction is in \cite{RS2011} where the authors determine $\mathrm{rk}(M)$ for the monomials
\[M=\left(x_1\cdot\ldots\cdot x_n\right)^m\] for any $n$ and $m$. In particular, they show that
$\mathrm{rk}(M)=(m+1)^{n-1}$.  In this paper we generalize that result and find  $\mathrm{rk}(M)$ for any monomial, thus completely solving the Waring Problem for monomials.

Our approach to solving the Waring Problem for specific
polynomials follows a well known path, namely the use of the
Apolarity Lemma \ref{apolarityLEMMA} to relate the computation of
$\mathrm{rk}(F)$ to the study of ideals of reduced points
contained in the ideal $F^\perp$, see Section \ref{mainresultSEC}.
For the special case in which $F$ is a monomial we get a bound on
the multiplicity of an ideal of reduced points $I\subset F^\perp$,
see Theorem \ref{mainthm}. Among the consequences of this result
is Corollary \ref{monmialsumofpoercorol} where we show that
\[\mathrm{rk}(x_1^{b_1}\cdot\ldots\cdot x_n^{b_n})=\prod_{i=2}^n(b_i+1),\]
where $1\leq b_1\leq \ldots \leq b_n$. Moreover, we describe an
explicit sum of powers decomposition for monomials, see Corollary
\ref{sumofpowerdecmon}. We also obtain a similar result for the
forms which are a sum of pairwise coprime monomials, see Corollary
\ref{coprimemonomialscor}.

\section{Basic facts}

We consider $k$ an algebraically closed field of characteristic
zero and the polynomial ring $T= \oplus_{i=0}^\infty T_i = k[y_1,\ldots,y_n]$. Given a homogeneous ideal
$I\subset T$ we denote by $$HF(T/I,i)= \dim_kT_i -\dim_k I_i$$ its {\it Hilbert
function} in degree $i$. It is well known that for all $i > > 0$ the function $HF(T/I,i)$ is a polynomial function with rational coefficients, called the {\it Hilbert polynomial} of $T/I$.  We say that an ideal $I\subset T$ is one
dimensional if the Krull dimension of $T/I$ is one, equivalently the Hilbert polynomial of $T/I$ is some integer constant, say $s$.  The integer $s$ is then called the {\it multiplicity} of $T/I$.  If, in addition, $I$ is a radical ideal, then $I$ is the ideal of a set of $s$ distinct points.  We will use the fact that if $I$ is a one dimensional saturated ideal of multiplicity $s$, then $HF(T/I, i)$ is always $\leq s$.

Let $S$ be another polynomial ring,  $S = \oplus _{i=0}^\infty S_i
= k[x_1,\ldots,x_n]$.  We make $S$ into a $T$-module by having the
variables of $T$ act as partial differentiation operators, e.g. we
think of $y_1 =\partial/ \partial x_1$. (see, for example,
\cite{IaKa} or \cite{Ge}).  We refer to a polynomial in $T$ as
$\partial$ instead of using capital letters. In particular, for
any form $F$ in $S_d$ we define the ideal $F^\perp\subseteq T$ as
follows:
\[F^\perp=\left\{\partial\in T : \partial F=0\right\}.\]

The following lemma, which we will call {\it Apolarity Lemma},  is a consequence of \cite[Lemma 1.15]{IaKa}.

\begin{lem}\label{apolarityLEMMA}
A homogeneous degree $d$ form $F\in S$ can be written as
\[F=\sum_{i=1}^s L_i^d , \ L_i \hbox{ pairwise linearly independent}\]
 if and only if there exists $I\subset F^\perp$ such that $I$ the ideal of a set of $s$ distinct points in $\PP^{n-1}$.
\end{lem}

We will also need the following fact.

\begin{lem}\label{HFlemma} For $n>1$ consider the ideal $J=(y_1,y_2^{a_2},\ldots,y_n^{a_n})\subset
T$, where $2\leq a_2\leq\ldots \leq a_n$.  Set $\tau=a_2+\ldots
+a_n-(n-1)$. Then,
\[\sum_{i=a_2}^\tau HF(T/J,i)=(\prod_{i=2}^n a_i) - {a_2+n-2 \choose n-1}.\]
\end{lem}
\begin{proof} We first note that $J = (y_1, y_2^{a_2}, \ldots , y_n^{a_n})$ is a homogeneous complete intersection ideal in $k[y_1, \ldots , y_n]$.  In this case, it is well known that
$$A = k[y_1, \ldots ,y_n]/J = \oplus_{i=0}^\infty A_i$$
is an Artinian Gorenstein ring for which:

\begin{enumerate}

\item if $\tau = ( a_2 + \cdots + a_n) - (n-1)$ then $\dim A_\tau \neq 0$  (in fact, its dimension is 1) and $\dim A_\ell =  0$ for $\ell > \tau$;

\item\label{ii} $\dim_kA = \sum_{i=0}^ \tau  \dim_kA_i = 1\cdot a_2\cdots a_n = \prod_{i=2}^n {a_i}$. Since $J_i = (y_1)_i$ for $0 \leq i \leq a_2 - 1$, it follows that (for these same $i$) we have
$$
\dim_kA_i = k[y_1, \ldots , y_n]_i - \dim_k k[y_1, \ldots , y_n]_{i-1} = {i+n-2\choose n-2}.
$$

\end{enumerate}

As a consequence we easily get that
$$
\sum_{i=0}^{a_2 -1} \dim_kA_i = \sum_{i=0}^{a_2 -1} {i+n-2\choose n-2} = {a_2 + n-2 \choose n-1} .
$$
Thus, rewriting \eqref{ii} above, we get
$$
\sum_{i=a_2}^\tau HF(T/J, i) = (\prod_{i=2}^n {a_i}) - {a_2 + n-2 \choose n-1}
$$
which is what we wanted to prove.
\end{proof}

We conclude with the following trivial, but useful, remark.

\begin{rem}\label{leastvarREM}{
The rank of a form $F$ can be computed in the
polynomial ring with the least number of variables containing $F$.
To see this, consider a rank $d$ form $F\in k[x_1,\ldots,x_n]$ and
suppose we know $\mbox{rk}(F)$. We can also consider $F\in
k[x_1,\ldots,x_n,y]$ and we can look for a sum of powers
decomposition of $F$ in this extended ring. If
\[F(x_1,\ldots,x_n)=\sum_1^r
\left(L_i(x_1,\ldots,x_n,y)\right)^d,\] then, by setting $y=0$, we
readily get $r\geq \mbox{rk}(F)$. Thus, by adding variables we can
not get a sum of powers decomposition involving fewer summands. In
particular, given a monomial
\[M=x_1^{b_1}\cdot\ldots\cdot
x_n^{b_n},\] with $1\leq b_1\leq\ldots\leq b_n$ it is enough to
work in $k[x_1,\ldots,x_n]$ in order to compute $\mbox{rk}(F)$.}\end{rem}

\section{Main result and applications}\label{mainresultSEC}

\begin{thm}\label{mainthm} Let $n>1$ and  $K=(y_1^{a_1},\ldots,y_n^{a_n})$ be an ideal of $T$ with $2\leq a_1\leq\ldots
\leq a_n$. If $I\subset K$ is a one dimensional radical ideal of
multiplicity $s$, then
\[s\geq \prod_{i=2}^n a_i.\]
\end{thm}
\begin{proof}
We consider two cases depending on whether or not $y_1$ is a zero
divisor in $T/I$. The condition $2\leq a_1$ is needed for
the latter, while the former also works for $a_1=1$.

Suppose that $y_1$ is not a zero divisor in $T/I$.  Consider the short
exact sequence
\[0\longrightarrow {T \over I }(-1) \stackrel{y_1}{\longrightarrow} {T \over I } \longrightarrow {T \over I+(y_1) } \longrightarrow 0.\]
Set $h_i=HF(T/I,i)$ and $J=(y_1,y_2^{a_2},\ldots,y_n^{a_n})$ as in Lemma 2.2.
Clearly
\[ I \subset I + (y_1) \subset K + (y_1) = J \]
and hence
\[h_i-h_{i-1}=HF\left({T \over I+(y_1)},i\right)\geq HF\left({T\over J},i\right).\]

Now, as in Lemma \ref{HFlemma} we set $\tau=(\sum_2^na_i)-(n-1)$.    Using Lemma \ref{HFlemma} and the inequality we just found on $h_i - h_{i_1}$, we compute a bound on $h_\tau$ as follows:
\[h_\tau\geq h_{\tau-1}+HF\left({T\over J},\tau\right)\geq \ldots \geq h_{a_2-1} + \sum_{i=a_2}^\tau HF\left({T\over J},i\right).\]

As $y_1$ is not a zero divisor, we conclude that $I_{a_2-1}=(0)$ and hence $h_{a_2 -1} = {a_2 -1 + n-1 \choose n-1}$.  Thus Lemma \ref{HFlemma} yields $\prod_2^n
a_i\leq h_{\tau}\leq s$ as we wanted to show.

Now suppose that $y_1$ is a zero divisor on $T/I$.   Consider the ideal $I^\prime = I:(y_1) \supset I$.  Since $I$ was a one-dimensional radical ideal the same is true for $I^\prime$.  Since $T/I^\prime$ is a homomorphic image of $T/I$ the multiplicity of $T/I^\prime$ (which we will denote by $s^\prime$) satisfies $s^\prime \leq s$.  Notice now that $I^\prime : (y_1) = I^\prime$ (again since $I$ is radical) and so $y_1$ is not a zero divisor on $T/I^\prime $.

Since
$$
I^\prime = I:(y_1) \subseteq K:(y_1) = (K^\prime) = (y_1^{a_2-1}, y_2^{a_2}, \ldots , y_n^{a_n} )
$$
we can apply the first part of the argument to $I^\prime \subset K^\prime$ (we are using that $a_1\geq 2$ in this case) and so we obtain that
$$
s^\prime \geq \prod_{i=2}^n a_i
$$
and that completes the proof of the Theorem.
\end{proof}

\begin{rem}\label{radicalREM}{
We notice that the hypothesis of $I$ being radical is necessary. In fact, $I=(y_1^{a_1},\ldots,y_{n-1}^{a_{n-1}})\subset K$ and $I$ is a one dimensional saturated ideal of multiplicity $\Pi_1^{n-1}a_i$. In particular, if $a_1<a_n$ the multiplicity of $I$ is less than $\Pi_2^{n}a_i$.
}\end{rem}

We now provide some applications of Theorem \ref{mainthm}.

\subsection{Rank of monomials}

We first deal with the sum of powers decomposition of monomials.
\begin{cor}\label{monmialsumofpoercorol}
For integers $m>1$ and $1\leq b_1\leq \ldots \leq b_m$ the
monomial
\[x_1^{b_1}\cdot\ldots\cdot x_m^{b_m}\]
is the sum of $\prod_{i=2}^m(b_i+1)$ power of linear forms and no
fewer.
\end{cor}
\begin{proof}
Using Remark \ref{leastvarREM} it is enough to consider the case $n=m$.  So, we can assume our monomial is $M = x_1^{b_1}\cdot\ldots\cdot x_n^{b_m}\in k[x_1,\ldots,x_n]$ and we set $\sigma=\prod_{i=2}^n(b_i+1)$. Applying the Apolarity Lemma to $M$ we find that the perp ideal of $M$ is
\[K=(y_1^{b_1+1},\ldots,y_n^{b_n+1}).\]

Let $I\subset K$ be an ideal of $s$ distinct points, i.e. a one dimensional radical ideal in $K$ of multiplicity $s$.
Applying Theorem \ref{mainthm} we get $s\geq \sigma$ and hence $M$
cannot be the sum of fewer than $\sigma$ powers of linear forms.

It remains to show that $M$ is the sum of $\sigma$ powers of linear
forms, i.e. we need to find an ideal of $\sigma$ distinct points inside $K$.  But it is easy to verify that
the ideal
$$
(y_2^{b_2+1} - y_1^{b_2+1}, y_3^{b_3+1} - y_1^{b_3+1}, \ldots , y_n^{b_n+1} - y_1^{b_n+1} )
$$
is such an ideal.
\end{proof}

It is easy to use the previous result to compute the rank of a generalized version of monomials.

\begin{cor}
For integers  $1\leq b_1\leq \ldots \leq b_m$ and for linearly
independent linear forms $L_1,\ldots,L_n$ consider the form
\[F=L_1^{b_1}\cdot\ldots\cdot L_n^{b_n}.\]
Then
\[\mathrm{rk}(F)=\Pi_{i=2}^n(b_i+1).\]
\end{cor}

\subsection{On the rank of the sum of relatively coprime monomials}

It is possible to use Theorem \ref{mainthm} to give bounds on the
rank of forms more general than monomials. For example, we have
the following
\begin{cor}\label{coprimemonomialscor}
Let $F$ be a form of degree $d$ which can be written
$$
F=M_1+\ldots +M_r
$$
where the $M_i$ are monomials of degree $d$ such that
$\mathrm{GCD}(M_i,M_j)=1$ for $i\neq j$. Then one has
\[\mathrm{rk}(M_i)\leq\mathrm{rk}(F)\]
for $i=1,\ldots,r$.

Moreover, if we let $M=\prod_{i=1}^r M_i$,  then
\[\mathrm{rk}(F)\leq\mathrm{rk}(M).\]
\end{cor}
\begin{proof}
Clearly $F^\perp\supseteq M_1^\perp\cap\ldots\cap M_r^\perp$.
Given $\partial\in F^\perp$, we notice that the non-zero monomial
of the form $\partial\circ M_i$ are linearly independent as they
are pairwise coprime. Hence we have
\[F^\perp= M_1^\perp\cap\ldots\cap M_r^\perp.\]
In particular, $F^\perp\subseteq M_i^\perp,i=1,\ldots,r$ and then
$\mathrm{rk}(M_i)\leq\mathrm{rk}(F)$.

It is straightforward to notice that
\[M^\perp\subseteq F^\perp\]
and thus $\mathrm{rk}(F)\leq\mathrm{rk}(M)$.  That completes the proof.
\end{proof}

%

\subsection{On the rank of the generic form}

It is well known, see \cite{AH95}, that for the generic degree $d$
form in $n+1$ variables $F$ one has
\[\mathrm{rk}(F)=\left\lceil{{d+n\choose d}\over n+1}\right\rceil.\]
However, the rank for a given specific form can be bigger or
smaller than that number. Moreover, it is trivial to show that every form of degree $d$ is a sum of ${d+n\choose d}$ $d^{th}$ powers of linear forms.  But, in
general, it is not known how big the rank of a degree $d$ form can be.

Using the monomials we can try to produce explicit examples of
forms having rank bigger than that of the generic form. We
give a complete description of the situation for the case of three variables.

\begin{cor} Let $n=3$ and $d>2$ be an integer. Then
\[
\mbox{max}\left\{\mbox{rk}(M): M \in S_d\mbox{ is a
monomial}\right\}= \left\{\begin{array}{ll}\left({d+1\over
2}\right)^2 & d\mbox{ is odd} \\ \\ {d\over 2}\left({d\over
2}+1\right) & d\mbox{ is even}\end{array}\right.
\]
and this number is asymptotically ${3\over 2}$ of the rank of the
generic degree $d$ form in three variables.
\end{cor}
\begin{proof}
We consider monomials $x_1^{b_1}x_2^{b_2}x_3^{b_3}$ with the
conditions $b_1\leq b_2\leq b_3$,
\[b_1+b_2+b_3=d\]
and we want to maximize the function $f(b_2,b_3)=(b_2+1)(b_3+1)$.
Considering $b_1$ as a parameter we are reduced to an optimization
problem in the plane where the constraint is given by a segment
and the target function is the branch of an hyperbola. For any
given $b_1$, it is easy to see that the maximum is achieved when
$b_2$ and $b_3$ are as close as possible to ${d-b_1 \over 2}$.
Also, when $b_1=1$ we get the maximal possible value. In
conclusion $\mathrm{rk}(M)$ is maximal for the monomial
\[ M=x_1x_2^{d-1\over 2}x_3^{d-1\over 2} (d \mbox{ odd}) \mbox{ or } M=x_1x_2^{d \over 2}x_3^{{d\over 2} - 1} (d\mbox{ even}).\]
Writing $d=6p+q$, with $0\leq q \leq 5$, and computing one easily
sees that the rank of the generic forms is asymptotically $6p^2$.
While the maximal rank of a degree $d$ monomial is asymptotically
$9p^2$ and the conclusion follows.

\end{proof}

\begin{rem} For $n=3$, we compare the behavior of the generic form and of the
biggest rank monomials in the following table
\[\begin{array}{c|c|c}
d & \mathrm{rk}(\mbox{ generic degree $d$ form}) &
\mathrm{max}_M\mathrm{rk}(M) \\ \hline
3 & 4 & 4 \\
4 & 5 & 6 \\
5 & 7 & 9 \\
6 & 10 & 12 \\
7 & 12 & 16
\end{array}\]
\end{rem}

\subsection{Sum of powers decomposition for monomials}

Since we now know the rank of any given monomial $M$, we can try to give a
  description of one of its sum of powers decompositions.

\begin{cor}\label{sumofpowerdecmon} For integers  $1\leq b_1\leq \ldots \leq b_m$ consider the monomial
\[M=x_1^{b_1}\cdot\ldots\cdot x_n^{b_n}.\]
Then
\[M=\sum_{j=1}^{\mathrm{rk}(M)}\left[\gamma_j \left(x_1+\epsilon_j(2)x_2+\ldots +\epsilon_j(n)x_n\right)\right]^d\]
where $\epsilon_1(i)\ldots,\epsilon_{\mathrm{rk}(M)}(i)$ are the
$(b_i+1)$-th roots of $1$, each repeated $\Pi_{j\neq i,1}(b_i+1)$
times, and the $\gamma_j$ are scalars.
\end{cor}
\begin{proof}

Another consequence of \cite[Lemma 1.15]{IaKa} allows
one to write a form as a sum of powers of linear forms.
If $I\subset M^\perp$ is an ideal of $s$ points, then
\[M=\sum_{j=1}^{s}\gamma_j \left(\alpha_j(1)x_1+\alpha_j(2)x_2+\ldots +\alpha_j(n)x_n\right)^d\]
where the $\gamma_j$ are scalars and $[\alpha_1:\ldots:\alpha_n]$
are the coordinates of the points having defining ideal $I$. Given
$M$ we can choose the following ideal of points
\[I=(y_2^{b_2+1} - y_1^{b_2+1}, y_3^{b_3+1} - y_1^{b_3+1}, \ldots , y_n^{b_n+1} - y_1^{b_n+1} )
.\]

It is straightforward to see that the points defined by $I$ have
coordinates
\[[1:\epsilon(2):\ldots:\epsilon(n)]\]
where $\epsilon(i)$ is a $(b_1+1)$-th root of $1$. Taking all
possible combinations of the roots of $1$ we get the desired
$\Pi_{i=2}^n(b_i+1)$ points and the result follows.

\end{proof}

To find an explicit decomposition for a given monomial is then
enough to solve a linear system of equation to determine the
$\gamma_j$. For example, in the very simple case of $M=x_0x_1x_2$,
we only deal with square roots of $1$ and we get:
\[x_0x_1x_2={1\over 24}(x_0+x_1+x_2)^3-{1\over 24}(x_0+x_1-x_2)^3-{1\over 24}(x_0-x_1+x_2)^3+{1\over 24}(x_0-x_1-x_2)^3.\]

%
%

\subsection{The variety of reducible forms}

Finally we give an application to the study of the variety of
forms which factor in a prescribed way. These varieties were first
introduced by Mammana in \cite{Mamma}.  They were also studied in some
detail in  \cite{CaChGe} and some particular examples considered in
  \cite{arrondobernardi}. More precisely, given
integers $d$ and $n$ we consider a partition $\lambda \vdash d$ , $\lambda = (d_1, d_2, \cdots, d_r)$ and define  the variety
\[\mathbb{X}_\lambda=\left\{ [F_{d_1}\cdot\ldots\cdot F_{d_r}] : F_{d_i}\in S_{d_i}\right\}\]
parameterizing forms of degree $d$ in $n$ variables factoring as a
product of forms of degree $d_i$. Let $V_{n,d}$ denote the
Veronese variety parameterizing $d$-th power of linear forms in
$n$ variables. We use $\sigma_r(V_{n,d})$ to denote the closure of the set of points on secant $\PP^{r-1}$'s to $V_{n,d}$.
\begin{cor}
For integers $n>1$ and $1\leq d_1\leq \ldots \leq d_n$ we have
\[\mathbb{X}_\lambda\subset\sigma_r(V_{n,d}),\]
for $r=\prod_2^n (d_i+1)$
\end{cor}
\begin{proof}
It is clear that the orbit of the monomial
$x_1^{d_1}\cdot\ldots\cdot x_n^{d_n}$ under the action of $GL(n)$
is dense in $\mathbb{X}_\lambda$. Hence the conclusion follows.
\end{proof}

%


\begin{thebibliography}{Mam54}

\bibitem[AB11]{arrondobernardi}
E.~Arrondo and A.~Bernardi.
\newblock On the variety parameterizing completely decomposable polynomials.
\newblock {\em J. Pure Appl. Algebra}, 215(3):201--220, 2011.

\bibitem[AH95]{AH95}
J.~Alexander and A.~Hirschowitz.
\newblock Polynomial interpolation in several variables.
\newblock {\em J. Algebraic Geom.}, 4(2):201--222, 1995.

\bibitem[CCG08]{CaChGe}
E.~Carlini, L.~Chiantini, and A.V. Geramita.
\newblock Complete intersections on general hypersurfaces.
\newblock {\em Michigan Math. J.}, 57:121--136, 2008.

\bibitem[Ger96]{Ge}
A.V. Geramita.
\newblock Inverse systems of fat points: {W}aring's problem, secant varieties
  of {V}eronese varieties and parameter spaces for {G}orenstein ideals.
\newblock In {\em The Curves Seminar at Queen's, Vol.\ X (Kingston, ON, 1995)},
  volume 102 of {\em Queen's Papers in Pure and Appl. Math.}, pages 2--114.
  Queen's Univ., Kingston, ON, 1996.

\bibitem[IK99]{IaKa}
A.~Iarrobino and V.~Kanev.
\newblock {\em Power sums, {G}orenstein algebras, and determinantal loci},
  volume 1721 of {\em Lecture Notes in Mathematics}.
\newblock Springer-Verlag, Berlin, 1999.

\bibitem[LM04]{LM}
J.M. Landsberg and L.~Manivel.
\newblock On the ideals of secant varieties of {S}egre varieties.
\newblock {\em Found. Comput. Math.}, 4(4):397--422, 2004.

\bibitem[LT10]{LandsbergTeitler2010}
J.~M. Landsberg and Zach Teitler.
\newblock On the ranks and border ranks of symmetric tensors.
\newblock {\em Found. Comput. Math.}, 10(3):339--366, 2010.

\bibitem[Mam54]{Mamma}
C.~Mammana.
\newblock Sulla variet\`a delle curve algebriche piane spezzate in un dato
  modo.
\newblock {\em Ann. Scuola Norm. Super. Pisa (3)}, 8:53--75, 1954.

\bibitem[RS11]{RS2011}
K.~Ranestad and F.~Schreyer.
\newblock On the rank of a symmetric form.
\newblock Preprint arXiv:1104.3648, 2011.

\end{thebibliography}

\end{document}